\documentclass[12pt]{article}
\usepackage{amsmath,amssymb,amsfonts,amsthm}
\newenvironment{myabstract}{\par\noindent
{\bf Abstract . } \small }
{\par\vskip8pt minus3pt\rm}
\newcounter{item}[section]
\newcounter{kirshr}
\newcounter{kirsha}
\newcounter{kirshb}
\newenvironment{enumarab}{\setcounter{kirshb}{1}
\begin{list}{(\arabic{kirshb})}{\usecounter{kirshb}} }{\end{list}}
\newtheorem{theorem}{Theorem}[section]

\newtheorem{corollary}[theorem]{Corollary}
\theoremstyle{definition}

\newtheorem{definition}[theorem]{Definition}
\def\R{\mathbb{R}}

\def\A{{\mathfrak{A}}}
\def\B{{\mathfrak{B}}}

\def\K{{\bf K}}
\def\K{{\bf K}}

\def\Rd{{\ Rd}}
\def\(R)RA{{\bf (R)RA}}

\def\R{\mathbb{R}}
\def\B{{\sf B}}

\def\w{{\sf w}}
\def\y{{\sf y}}
\def\g{{\sf g}}

\def\r{{\sf r}}
\def\K{{\sf K}}

\def\R{\sf R}

\def\set#1{\{#1\} }

\def\A{{\mathfrak{A}}}
\def\B{{\mathfrak{B}}}

\def\A{{\mathfrak{A}}}
\def\B{{\mathfrak{B}}}

\def\Rd{{\mathfrak{Rd}}}

\def\R{\mathfrak{R}}

\def\pa{$\forall$}
\def\pe{$\exists$}

\def\set#1{ \{#1\}}

\def\pe{$\exists$}
\def\pa{$\forall$}

\def\w{{\sf w}}
\def\g{{\sf g}}
\def\y{{\sf y}}
\def\r{{\sf r}}

\def\ws{winning strategy}

\def\y{{\sf y}}
\def\g{{\sf g}}
\def\r{{\sf r}}
\def\w{{\sf w}}

\title{There is no finite variable universal axiomatization for any class $\K$ between diagonal free cylindric algebras and 
polyadic equality algebras of finite dimension $\geq 3$}
\author{Tarek Sayed Ahmed\\
Department of Mathematics, Faculty of Science,\\ 
Cairo University, Giza, Egypt.
  }
\begin{document}
\maketitle
\begin{myabstract} Using a rainbow construction for polyadic algebras, lifting a rainbow construction for relation algebras by Hirsch and Hodkinson, 
we prove the result in the title.
For each finite $n\geq 3$ we construct two finite rainbow polyadic algebras, such that for one \pa\ has a \ws\ so 
its diagonal free reduct is not representable, and for the other algebra \pe\ has a \ws\ 
so the algebra is representable. The algebras are based, as usual in cylindric rainbow constructions, on  
coloured graphs, labelled by the rainbow colours,  except that the representable algebra has more reds. 
All other colours are the same. Finally, we show that $n$ variable equations cannot 
see that one is representable (as a polyadic algebra) 
while the  diagonal free reduct of the other is not.

\end{myabstract}

\begin{theorem} Let $V$ be a discriminator variety. Assume that there are simple algebras $\A$ and $\B$ such that $\A\in V$ and $\B\notin V$, 
and for any equation involving $n$ variables, 
$\A$ falsifies $e$ if and only if $\B$ falsifies $e$.
Then $V$ is not axiomatizable by any set of prenex universal sentences that uses only $n$ variables.
\end{theorem}
\begin{proof} If $\Sigma$ is any $n$ variable equational theory then $\A$ and $\B$ either both validate 
$\Sigma$ or neither do. Since one algebra is in $V$ while the other is not, it follows that $\Sigma$ does not axiomatize $V$. 
If $\A$ and $\B$ are simple, then they
are subdirectly irreducible. In a dicriminator variety evey universal prenex formula is equivalent in subdirectly irreducible members to an 
equation using the same number of vaiabales. Hence the desired.
\end{proof}

For every $n\geq 3$, two  finite rainbow polyadic algebras are constructed, one will be representable, the diagonal free reduct 
of the other will not be representable. All colours are the same 
except that one has more red graphs than the other (a red graph is a coloured graph such that at least one edge is coloured red).
In the usual atomic game on graphs using his excess of greens, \pa\ wins. 
This prohibits the first algebra to be representable. In the second case the reds are more, and \pe\ can win the 
$\omega$ rounded game (in finitely many rounds) on coloured graphs, hence the algebra will be representable.
Futhermore, $n$ variable (diagonal free) equations cannot distinguish the two algebras. 

This is the general idea, now for the details:

Let $\kappa$ be a finite number $>n$. 

Let $\alpha=3.2^n$ and $\beta=(\alpha+1)(\alpha+2)/2.$

\begin{definition}
The colours we use:
\begin{itemize}

\item greens: $\g_i$ ($1\leq i<n-2)\cup \{\g^0_i: i\leq \alpha+2\}$,
\item whites : $\w_i, i <n$
\item yellow : $\y$
\item reds:  $\r_{i}$, $i\in \kappa$

\item shades of yellow : $\y_S: S\subseteq \alpha+2$

And coloured graphs are:
\begin{definition}
\begin{enumarab}

\item $M$ is a complete graph.

\item $M$ contains no triangles (called forbidden triples)
of the following types:

\vspace{-.2in}
\begin{eqnarray}
&&\nonumber\\
(\g, \g^{'}, \g^{*}), (\g_i, \g_{i}, \w),
&&\mbox{any }i\in n-1\;  \\
(\g^j_0, \g^k_0, \w_0)&&\mbox{ any } j, k\in \alpha+2\\
\label{forb:pim}(\g^i_0, \g^j_0, \r_{kl})&&\\
\label{forb:match}(\r_{i}, \r_{i}, \r_{j})&&
\end{eqnarray}
and no other triple of atoms is forbidden.

\item If $a_0,\ldots   a_{n-2}\in M$ are distinct, and no edge $(a_i, a_j)$ $i<j<n$
is coloured green, then the sequence $(a_0, \ldots a_{n-2})$
is coloured a unique shade of yellow.
No other $(n-1)$ tuples are coloured shades of yellow.

\item If $D=\set{d_0,\ldots  d_{n-2}, \delta}\subseteq M$ and
$\Gamma\upharpoonright D$ is an $i$ cone with apex $\delta$, inducing the order
$d_0,\ldots  d_{n-2}$ on its base, and the tuple
$(d_0,\ldots d_{n-2})$ is coloured by a unique shade
$y_S$ then $i\in S.$

\end{enumarab}
\end{definition}

\end{itemize}
\end{definition}
The set of coloured graphs (based on the above colours) are defined the usual way. Denote this class by 
$\K$. The atoms of the finite algebra are coloured graphs; however, we consider also the polyadic operations defined on the atom structures
consisting of equivaelence classes $[f]$ where $f: n\to \Gamma\in \K$, easily by 
$[f]_{ij}[g]$ iff $f=g\circ [i,j]$. This is well defined.
The algebras wil be finite because $n$ is finite and the colours are finite.

$\A=\A_{\alpha+2,\beta}$ and $\B=\A_{\alpha+2,\alpha}$, here $\alpha+2$ is the number of greens

\begin{definition}
The  game builds a nested sequence $M_0\subseteq M_1\subseteq \ldots $.
of coloured graphs.  
The game is $\omega$ rounded, but of course because our algebras are finite, the game
wil terminate in finitely many rounds.
\pa\ picks a graph $M_0\in \K$ with  
$|M_0|=n$. \pe\ makes no response
to this move. In a subsequent round, let the last graph built be $M_i$.
$\forall$ picks 
\begin{itemize}
\item a graph $\Phi\in \K$ with $|\Phi|=n$
\item a single node $k\in \Phi$
\item a coloured garph embedding $\theta:\Phi\sim \{k\}\to M_i$
Let $F=\phi\smallsetminus \{k\}$. Then $F$ is called a face. 
\pe\ must respond by amalgamating
$M_i$ and $\Phi$ with the embedding $\theta$. In other words she has to define a 
graph $M_{i+1}\in C$ and embeddings $\lambda:M_i\to M_{i+1}$
$\mu:\phi \to M_{i+1}$, such that $\lambda\circ \theta=\mu\upharpoonright F.$
\end{itemize} 
\end{definition}

\begin{theorem} \pa\ has a \ws\  for $\B$  in $\alpha+2$ rounds; hence $\Rd_{df}B\notin {\sf RDf_n}$
\end{theorem}
\begin{proof}
\pa\ plays a coloured graph $M \in \K$ with
nodes $0, 1,\ldots, n-1$ and such that $M(i, j) = \w (i < j <
n-1), M(i, n-1) = \g_i ( i = 1,\ldots, n), M(0, n-1) =
\g^0_0$, and $ M(0, 1,\ldots, n-2) = \y_{\alpha+2}$. This is a $0$-cone
with base $\{0,\ldots , n-2\}$. In the following moves, \pa\
repeatedly chooses the face $(0, 1,\ldots, n-2)$ 
and demands a node 
$t<\alpha+2$ with $\Phi(i,\alpha) = \g_i (i = 1,\ldots,  n-2)$ and $\Phi(0, t) = \g^t_0$,
in the graph notation -- i.e., a $t$ -cone on the same base.
\pe\, among other things, has to colour all the edges
connecting nodes. By the rules of the game 
the only permissible colours would be red. Using this, \pa\ can force a
win in $\alpha+2$ rounds eventually using her enough supply of greens, 
which \pe\ cannot match using his $<$ number of reds. The conclusion now follows since $\B$ is generated by elements whose dimension sets 
are $<n$.
\end{proof}

\begin{theorem} The algebra $\A\in {\sf RPEA_n}$
\end{theorem}

\begin{proof} If \pa\ plays like before, now \pe\ has more  reds, so \pa\ cannot force a win. In fact \pa\ 
can only force a red clique of size $\alpha+2$, not bigger. So \pe\ s 
startegy within red cliques is to choose a label for each edge using  
a red colour and to ensure that each edge within the clique has a label unique to this edge (within the clique). 
Since there are $\beta$ many reds she can do that.

Let $M$ be a coloured graph built at some stage, and let \pa\ choose the graph $\Phi$, $|\Phi|=n$, then $\Phi=F\cup \{\delta\}$,
where  $F\subseteq M$ and $\delta\notin M$. 
So we may view \pa\ s move as building a coloured graph $\Gamma^*$ extending $M$
whose nodes are those of $\Gamma$ together with $\delta$ and whose edges are edges of $\Gamma$ together with edges
from $\delta$ to every node of $F$. 

Colours of edges and $n-1$ tupes in $\Gamma^*$ but not
in $M$ are determined by \pa\ moves. 
No $n-1$ tuple containing both $\delta$ and elements of $M\sim F$ 
has a colour in $\Gamma^*$

Now \pe\ must extend $\Gamma^*$ to a complete the graph on the same nodes and 
complete the colouring giving  a graph $\Gamma^+$ in $\K$. 
In particular, she has to define $\Gamma^+(\beta, \delta)$ 
for all nodes  $\beta\in M\sim F$.
\begin{enumarab}
\item  if $\beta$ and $\delta$ are both apexes of cones on $F$, that induces the same linear ordering on $F$, the 
\pe\ has no choice but to pick a  red atom, and as we desribed above, she can do that 

\item Other wise, this is not the case, so for some $i<n-1$ there is no $f\in F$ such 
that $\Gamma^*(\beta, f), \Gamma^* (f,\delta)$ are both coloured $\g_i$ or if $i=0$, they are coloured
$\g_0^l$ and $\g_0^{l'}$ for some $l$ and $l'$.
\end{enumarab}
In the second case \pe\ uses the normal strategy in rainbow constructions. 
She chooses $\w_0$, for $\Gamma^+(\beta,\delta)$.

Now we turn to coluring of $n-1$ tuples. For each tuple $\bar{a}=a_0,\ldots a_{n-2}\in \Gamma^{n-1}$ with no edge 
$(a_i, a_j)$ coloured green, then  \pe\ colours $\bar{a}$ by $\y_S$, where
$$S=\{i\in \alpha+2: \text { there is an $i$ cone in $\Gamma^*$ with base $\bar{a}$}\}.$$
This works.
\end{proof}

\begin{theorem} A coloured graph is red, if at least one of its edges are labelled red
\end{theorem}
We write $\r$ for $a:n\to \Gamma$, where $\Gamma$ is a red graph, and we call it a red atom. 
(Here we identify an atom with its representative, but no harm will follow).

\begin{theorem} For any $n$ variable equation the two algebras $\A$ and $\B$,  falsify it together or satisfy it together.
\end{theorem}
\begin{proof}  
Let $\R$ be the set of red graphs of $\A$, and $\R'$ be the set of red graphs in $\B$. Then $|\R|\geq |\R'|\geq 3.2^n$.
Assume that the equation $s=t$, using $n$ variables does not hold in $\A$. 
Then there is an assignment $h:\{x_0,\ldots x_{n-1}\}\to \A$, such that $\A, h\models s\neq t$. 
We construct an assignment $h'$ into $\B$ that also falsifies $s=t$. 
Now $\A$ has more red atoms, but $\A$ and $\B$ have identical non-red atoms. So for any non red atom $a$ of $\B$, and for any
$i<n$, let $a\leq h'(x_i)$ Iff $a\leq h(x_i)$.
The asignment $h$ induces a partition of $\R$ into $2^n$ parts $\R_S$, $S\subseteq \{0,1\ldots n-1\}$, by
$$\R_S=\{\r: \r\leq h(x_i) , i\in S, \r.h(x_i)=0, i\in n\sim S\}.$$
Partition  $\R'$ into $2^n$ parts: $S\subseteq n$ such that $|\R'_S|=|\R_S|$ if $|\R_S|<3$, and $|\R_S'|\geq 3$ iff $|\R_S|\geq 3$.
This possible because $|\R|\geq 3.2^n$.
Now for each $i<n$ and each red atom $r'$ in $\R'$, we complete the definition of
$h'(x_i)\in \B$  by
$r'\leq h'(x_i)$ iff $r'\in \R'_S$ for some $S$ such that $i\in S$.
This can be easily checked to satisfy the required. The converse is entirely analogous.
\end{proof}
\begin{corollary} For $n\geq 3$, the class of representable algebras in any $\K$ 
between ${\sf Df}_n$ and ${\sf PEA_n}$ 
does not have a prenex universal axiomatization using $n$  variables.
\end{corollary}
\end{document}